\documentclass[a4paper,notitlepage,twoside,reqno,11pt]{amsart}

\usepackage{mathrsfs}
\usepackage{yhmath}
\usepackage{tikz}
\usepackage[hmargin=1in,vmargin=1in]{geometry}
\usepackage{amsmath}
\usepackage{amssymb}
\usepackage{amsfonts}
\usepackage{amsthm}
\usepackage{bbm,pifont,latexsym}
\usepackage{anysize}  
\marginsize{3cm}{3cm}{3cm}{3cm}

\usepackage{dcolumn,indentfirst,color,hyperref}
\usepackage{amsmath,amssymb,amscd,amsthm,amsfonts,mathrsfs}
\usepackage{color,graphicx,xcolor,graphics}
\usepackage{titlesec} 
\usepackage{titletoc} 
\titlecontents{section}[20pt]{\addvspace{2pt}\filright}
              {\contentspush{\thecontentslabel. }}
              {}{\titlerule*[8pt]{}\contentspage}

\usepackage{fancyhdr} 
\newtheorem{thm}{Theorem}[section]

\newtheorem{lem}[thm]{Lemma}
\newtheorem{cor}[thm]{Corollary}
\newtheorem{prop}[thm]{Proposition}

\newtheorem{defi}{Definition}[section]
\newtheorem{conj}{Conjecture}[section]

\newcommand{\wt}{\widetilde}
\newcommand{\mb}{\mathbb}
\newcommand{\mc}{\mathcal}
\newcommand{\sm}{\setminus}
\newcommand{\tu}{\textup}
\newcommand{\ol}{\overline}
\newcommand{\es}{\emptyset}

\newcommand{\tb}{\textbf}

\newcommand{\wh}{\widehat}

\newcommand{\olD}{\overline{\mathbb{D}}}

\makeatletter\@addtoreset{equation}{section}\makeatother 
\titleformat{\section}{\centering\normalsize}{\textsc{\thesection.}}{1em}{\textsc}
\titleformat{\subsection}{\normalsize}{\thesubsection.}{1em}{\textbf}

\title
 {Most Fatou and Julia components are small for polynomials}
\begin{document}

\author{Jinsong Zeng}
\address{Jinsong Zeng, School of Mathematical Sciences, Shenzhen University, Shenzhen 518060, P. R. China}
\email{jinsongzeng@163.com}

\begin{abstract} We prove that Julia components of polynomials are generally small in diameter. For polynomials without irrationally neutral cycles, Fatou components are also typically small, even when the Julia set is not locally connected.
 \end{abstract}



\maketitle
\section{Introduction}

We study the dynamics of polynomials $f:\mb{C}\to\mb{C}$ of degree $d\geq 2$. A well-known result regarding the Branner-Hubbard Conjecture states that non-trivial Julia components of polynomials are eventually periodic\cite{KS,QY}. These countably many components (if any) are everywhere dense in the Julia set. We prove that most Julia components are small. Precisely,

\begin{thm}\label{th:Julia}
Let $f:\mathbb{C}\to\mathbb{C}$ be a polynomial. For any $\epsilon>0$, there are at most finitely many Julia components of $f$ with diameters greater than $\epsilon$.
\end{thm}
Theorem \ref{th:Julia} is not true for rational maps. McMullen provided an example of a rational Julia set homeomorphic to a Cantor set of circles, containing uncountably many components with diameters bounded below by some $\epsilon_0>0$ \cite{Mc}.

For polynomials, Roesch and Yin showed that every bounded Fatou component not eventually onto a Siegel disk is a Jordan domain \cite{RY}. While Siegel disks, i.e., Fatou components on which the map is conjugate to irrational rotations, are conjectured to be Jordan domain. We focus on the asymptotic shrinking of non-Siegel Fatou components. The main result is that 
\begin{thm}\label{th:Fatou}
Let $f:\mathbb{C}\to \mathbb{C}$ be a polynomial without irrationally neutral cycles. For any $\epsilon>0$, there are at most finitely many Fatou components of $f$ with diameters greater than $\epsilon$.
\end{thm}
Kiwi \cite{Ki} proved that for polynomials with connected Julia sets and no irrationally neutral cycles, local connectivity of the Julia set is equivalent to the absence of wandering non-trivial connected Julia subsets. Infinitely renormalizable polynomials with connected but non-locally connected Julia sets exist \cite{Mlc}; Theorem \ref{th:Fatou} demonstrates that most Fatou components remain small despite such complexities.


The local connectivity of Julia sets is crucial in holomorphic dynamics. Partly because the dynamics on locally connected Julia sets can be accurately described via symbolic dynamics. It is a fruitful area of research on polynomials, see \cite{Mlc,Pe,KL2,KS,Le,LG,Ya,CST} and others.

Roesch obtained the breakthrough result for rational maps in this area. She proved that except for some special cases, cubic Newton maps have locally connected Julia sets. This includes some maps that are infinitely renormalizable or contain Cremer points, unlike the case for polynomials \cite{Roe}. These findings have been generalized to high degree Newton maps, as discussed in \cite{DS,RYZ}. In \cite{QWY} they showed that the Julia sets of McMullen maps are locally connected in many cases.

Recently, it has been proven in \cite{WYZZ, FY} that the Julia sets of a number of rational maps with bounded type Siegel disks are locally connected. Their approach goes by analyzing the contraction property when pulling back certain disks near the boundaries of the bounded type Siegel disks.

Theorem \ref{th:Fatou} has significant applications in the study of local connectivity of rational Julia sets. Specifically, a renormalizable rational map usually contains copies of polynomial Fatou components. By Whyburn's criterion, to establish the local connectivity, it is necessary to demonstrate that the majority of these components are small. As a consequence of Theorem \ref{th:Fatou}, Julia sets of Newton maps without irrationally neutral cycles are locally connected. Compare \cite[Corollary 1.2]{DS}.

A key point in the proof of Theorem \ref{th:Fatou}, as stated in Corollary \ref{cor:disjoint}, is the disjointness of bounded Fatou components and non-trivial \emph{ends} $E$ in their boundaries. This guarantees the finiteness of hyperbolic diameters of Fatou components with respect to the domain $\mb{C}\sm E$. Nevertheless, when $f$ has irrationally neutral cycles, this property no longer holds.

In the case of a cubic Newton map $f$ that is exactly once renormalizable and has a Cremer point, similar difficulties occur: the components of root basins, acting as ``carrots", intersect the quadratic Julia set copy at the iterated preimages of the  $\beta$-fixed point in the boundary, resulting in each carrot having an infinite hyperbolic diameter. Intuitively, the shrinking of carrots depends on the dynamic behavior of the critical point in the Julia set copy. In the proof of \cite[Corollaries 1.2 and 1.3]{DS} the convergence of carrot diameter does not seem to be discussed. Generally, we propose the following conjecture.

\begin{conj}
	Let $f:\mathbb{C}\to\mathbb{C}$ be a polynomial. Then there are at most finitely many Fatou components of $f$ with diameters greater than $\epsilon$ for any $\epsilon>0$.
\end{conj}
There are many parallel phenomenons in parameter spaces. For instance, fix a strictly copy $\mc{M}_0$ of the Mandelbrot set $\mc{M}$, components of $\mathcal{M}\sm \mathcal{M}_0$ are referred to as \emph{carrots}; each of them intersects $\mc{M}_0$ at a dyadic tip of $\mc{M}_0$ in its closure \cite{DH2,PR}. Petersen-Roesch and Dudko independently established the shrinking of these carrots in size \cite{Du,PR}.
\vskip 0.3cm
\emph{Sketch of the proof of Theorem \ref{th:Fatou}.} Let $\{X_n\}$ be a sequence of the closure of bounded Fatou components that converges to a continuum $X_\infty$ in the Hausdorff topology. Without loss of generality, we assume that the Julia set is connected and periodic Fatou components are fixed.

We begin by constructing the Branner-Hubbard-Yoccoz puzzles from invariant graphs, which are mainly composed of external rays landing at biaccessible fixed points and rays in bounded Fatou components. The infinite intersections of nested closed puzzle pieces decompose the entire Julia set into dynamical invariant ends, namely satellite and primitive ends. The limit set $X_\infty$ must be contained in one of these ends.

We apply standard distortion lemmas for non-persistently recurrent critical ends and well-known results on the existence of complex bounds for persistently recurrent critical ends to show the triviality of wandering ends. Consequently, we can assume that $X_\infty$ belongs to some periodic non-trivial end $\tb{e}$.

We find two strictly nested (thickened) puzzle pieces $P_0$ and $P_1$ containing $\tb{e}$ and study the dynamics of $g=f^k:P_1\setminus\tb{e}\to P_0\setminus\tb{e}$ restricted to the two annuli. A key point is that all $X_n$ are disjoint from $\tb{e}$ under the assumption of the absence of irrationally neutral cycles. Therefore, the hyperbolic diameter of each $X_n$ is finite. The weaker expansivity of $g$ would imply the uniform boundedness of these hyperbolic diameters. As a result, the Euclidean diameter has to approach zero.

\section{Puzzles and ends}
In this section, we will construct a specific Branner-Hubbard-Yoccoz puzzle for polynomials and discuss the nested properties of puzzles, as well as the dynamics around periodic ends.

As a start, we recall some classical facts from \cite{DH1}; See also \cite{M1}. Let $f:\mb{C}\to\mb{C}$ be a polynomial of degree $d\geq 2$. The \emph{filled Julia set} is defined as $$K_f=\{z\in\mb{C}: f^n(z)\tu{ is bounded}\},$$ whose boundary is the Julia set $J_f$. The \emph{basin of infinity} $B_\infty$ is the complement of $K_f$. When $K_f$ is connected, the polynomial $f$ is conjugated to $z\mapsto z^d$ on $B_\infty$. The conjugacy $\phi:B_\infty\to \mb{C}\sm \mb{D}$ tangent to the identity at infinity is called the \emph{B\"{o}ttcher coordinate} of $f$ on $B_\infty$. 

The \emph{external ray} of argument $\theta$ for $B_\infty$ is defined as $R(\theta)=\phi^{-1}(]1,\infty[e^{2\pi i \theta})$. Every periodic external ray
lands, i.e., $\lim_{r\to 1}\phi(re^{2\pi i\theta})$ exists. The landing point is either repelling or parabolic. Conversely, a repelling or parabolic point receives at least one and a finite number of external rays, all of which are necessarily periodic of the same period.  The preimage $\phi^{-1}(\{e^{r+2\pi i\theta}:0\leq \theta\leq 2\pi\})$ is called the \emph{equipotential curve} of level $r>0$.  
 
Associated to a periodic non-Siegel Fatou component,  P. Roesch and Y. Yin provided a decomposition of the entire filled Julia set by limbs \cite{RY}.
   
\begin{thm}{\cite[Theorems A and B]{RY}}\label{thm:RY}
	Let $f$ be a polynomial with connceted Julia set. Let $U$ be a bounded periodic Fatou component of $f$ that is either attracting or parabolic. Then $\partial U$ is a Jordan curve. Moreover, the filled Julia set $K_f=\ol{U}\cup\bigsqcup_{t\in \mathbb{S}^1}L_t$ with the following properties:
	
	\tu{(1)} $L_t$ is compact and connected;
	
	\tu{(2)} $L_t\cap \ol{U}$ is one point, denoted by $\gamma(t)$;
	
	\tu{(3)} If $L_t\neq \gamma(t)$, two external rays land at $\gamma(t)$ that separate $L_t$ from $\ol{U}$; In this case, $L_t$ is called a \emph{limb} of $K_f$ and $\gamma(t)$ is called the \emph{root} of the limb. 
	
	\tu{(4)} If $L_t\neq \gamma(t)$, then $L_t$ is eventually bijectively mapped onto a limb that contains a critical point of $f$.
\end{thm}
\subsection{Puzzles}\label{subsection}
Throughout this section, we assume $f$ is a polynomial of degree $d\geq 2$ having connected Julia set $J_f$ and without irrationally neutral cycles; By replacing $f$ by certain iterate $f^n$, we adopt the following assumptions:

(A1) periodic Fatou components are fixed by $f$;

(A2) periodic roots are fixed by $f$; They are finitely many by Theorem \ref{thm:RY}.

To create puzzles, we introduce a graph $\Gamma$ that partitions $J_f$ into pieces. Let $U$ be a fixed bounded Fatou component.

When $U$ is attracting, take a Jordan disk $D\subset U$ containing the only fixed point in $U$ such that $f(D)\Subset D$, and $\partial D$ avoids the critical orbits. Then a component of $(f|_U)^{-N}(D)$, for a large integer $N$, contains all critical points in $U$. This component, denoted by $D_N$, is a Jordan disk in $U$.

By a classic result from complex dynamics, there exists an arc $\gamma:[0, 1)\to U\setminus f(D_N)$ such that $\gamma\subset f(\gamma)$ and $\gamma$ joins $\partial f(D_N)$ to a fixed point (i.e., $\tu{lim}_{t\to1}\gamma(t)$ exists).  We define $$\Gamma_U=\partial D_N\cup \ol{(f|_U)^{-1}(\gamma)},$$
where $(f|_U)^{-1}(\gamma)$ consists of exactly $\tu{deg}(f|_U)$ arcs in $U\sm D_N$. 

When $U$ is parabolic, take an attracting petal $D\subset U$ of the parabolic point in $\partial U$ and ensure that $\partial D$ avoids critical orbits in $U$. For a large $N$, a component of $(f|_U)^{-N}(D)$ covers all critical points in $U$, which is denoted by $D_N$. We define $\Gamma_U=\partial D_N$. In this case, $\Gamma_U$ is a Jordan curve that intersects $\partial U$ in a maximum of $\mathrm{deg}(f|_U)^N$ pre-parabolic points.

Let $E$ be an equipotential curve around the filled Julia set $K_f$. We may define the graph as follows
\begin{equation}\label{eq:gamma}
\Gamma=E\cup \bigcup \Gamma_U\bigcup \ol{R},
\end{equation}
where $R$ runs over all external rays that land at $\bigcup\Gamma_U$ or \emph{biaccessible} fixed points. By a \emph{biaccessible} point we mean a point receiving at least two external rays.  Note that under Assumptions (A1) and (A2), periodic roots are biaccessible and contained within $\Gamma$, with external rays landing at them remaining fixed.
  
\begin{defi}
	The \emph{puzzle pieces} of \emph{depth} $n$ are the connected components of $\mb{C}\sm f^{-n}(\Gamma)$ that intersect $J_f$ for any $n\geq 0$. 
\end{defi}
By construction, puzzle pieces are Jordan domains, and two distinct puzzle pieces are either disjoint (that is, $P\cap Q=\es$) or nested (that is, $P\subset Q$ or $Q\subset P$). 
In particular, when $P\subset Q$, then $P\Subset Q$ if and only if $\partial P$ and $\partial Q$ have no common points in $J_f$, since rays do not cross.

\begin{lem}\label{lem:puzzles}
	Let $P$ be a puzzle piece of depth $\geq 1$, whose boundary contains a fixed point $z$, such that $P\subset f(P)$. Then there is a small open disk $\Delta\subset P$ with $\partial \Delta$ formed by 
	
	\tu{(1)} an arc $\alpha$ in $\partial P$ passing through $z$, and
	
	\tu{(2)} another arc in $P$ joining the two endpoints of $\alpha$,
	
	\noindent such that $f(\partial \Delta)\cap \Delta=\es$. 	
\end{lem}
\begin{proof}
	The fixed point $z$ is repelling or parabolic. All external rays landing at $z$ are fixed and contained in $\Gamma$. If $z$ is repelling, in a disk neighborhood $\mc{N}$ of $z$ with $f(\mc{N})\subset \mc{N}$, the graph $\Gamma\cap \mc{N}$ is $f$-invariant in the sense that $$\mc{N}\cap\Gamma\subset f(\mc{N}\cap \Gamma).$$
	Let $\Delta$ be the component of $\mc{N}\sm \Gamma$ in $P$ whose boundary contains $z$. Then $\Delta$ is as required.
	
	If $z$ is parabolic, let $\mc{N}$ be a small neighborhood of $z$ such that $\mc{N}\sm\{z\}$ is formed by repelling and attracting petals; See \cite[Theorem 10.7]{M1}. We may require that the petal in an immediate attracting basin $U$ of $z$ is bounded by $f^n(\Gamma\cap U)$ for a large $n$. 
	
	Claim that the disk $\Delta=\mc{N}\cap P$ is as required. Indeed, two immediate attracting basins of $z$, if they exist, are separated by a pair of fixed external rays landing at $z$. Thus $\Delta$ belongs to one of the repelling petals of $\mc{N}$ whose boundary contains the arc $\beta=\partial \Delta\cap P$. Then $f(\beta)\cap \Delta=\es$. Let $\alpha=\partial\Delta\sm \beta\subset \partial P$. Then $f(\alpha)$ is contained in the boundary of the puzzle piece $f(P)$. Since $P\subset f(P)$, we have $f(\alpha)\cap P=\es$. Thus $f(\alpha)\cap \Delta=\es$. The proof is complete.  
\end{proof}

A puzzle piece is always properly mapped onto another puzzle piece with one depth less through the map $f$. If the two puzzle pieces are nested, we can determine the degree by counting the number of fixed points.

\begin{lem}\label{lem:fixed-points}
	Let $P$ be a puzzle piece of depth $\geq 1$. If $P\subset f(P)$, then $\ol{P}$ contains exactly $\tu{deg}(f|_P)$ distinct fixed points.
\end{lem}
\begin{proof}
	The proof is based on the Lefschetz Fixed Point Theorem, which states that the degree of a weakly polynomial-like map, with isolated fixed points, is equal to the sum of Lefschetz indices at fixed points; See \cite[Lemma 3.7]{GM}. 
	
	Note that the map $g=f:\ol{P}\to f(\ol{P})$ is a weakly polynomial-like map, with isolated fixed points; See \cite[Definition 3.6]{GM}. We claim that the Lefschetz index $\iota(g, z)$ of each $g$-fixed point $z$ equals $+1$. In fact, an interior fixed point in $P$ cannot be parabolic, then its Lefschetz index is $+1$.  In the case of a boundary fixed point $z\in \partial P$, then $z$ is repelling or parabolic. In both cases, there is a degree one weakly polynomial-like map $f:\ol{\Delta}\to f(\ol{\Delta})$ according to Lemma \ref{lem:puzzles}. The Lefschetz Fixed Point Theorem implies that $\iota(f|_{\ol{\Delta}}, z)=+1$. Since the Lefschetz index is local invariant, we have $\iota(g, z)=+1$ as well.
	
	Since all fixed points of $g$ have Lefschetz indices $+1$, the number of fixed point equals $\tu{deg}(g)$ by the Lefschetz Fixed Point Theorem. The proof of the lemma is complete.
\end{proof}

\subsection{Ends}
We now study the infinite intersections of nested closed puzzle pieces, which divide the whole Julia sets into dynamic invariant full continua, known as ``ends". 
\begin{defi}
	Let $Z$ be the collection of fixed points in $\Gamma$ (see \eqref{eq:gamma} for $\Gamma$) and $Z_\infty=\bigcup_{n\geq 0} f^{-n}(Z)$. For $z\in Z_\infty$, we denote
	$Y_n(z)=\bigcup\ol{P},$
	where $P$ runs over all depth-$n$ puzzle pieces whose closures contain $z$. The \emph{satellite end} of the point $z$ is defined as $$\tb{e}(z)=\bigcap_{n\geq 0} Y_n(z).$$
	If $z \in J_f\sm Z_\infty$ is not contained in any satellite end, its \emph{primitive end} is defined as $$\tb{e}(z) = \bigcap_{n \geq 0}\ol{P_n(z)},$$
	where $P_n(z)$ is the unique depth-$n$ puzzle piece containing $z$. By an \emph{end} we mean a satellite or primitive end.
\end{defi}
Remarkably, the number of depth-$n$ puzzles in $Y_n(z)$ remains constant regardless of $n$, as long as $n$ is sufficiently large. Moreover, $Y_n(z)$ is a closed Jordan disk containing $z$ in its interior, whenever $z\in Z$ is repelling; If $z\in Z$ is parabolic, then $z\in \partial Y_n(z)$ for all $n\geq0$.

From the construction of $\Gamma$, ends are \emph{full continua}, i.e., compact, connected, and having only one complementary component, all of which are contained in the Julia set. Ends also possess other basic properties as stated in the following lemma.

\begin{lem}\label{lem:ends}
	\tu{(1)} the map $f$ sends satellite and primitive ends onto satellite and primitive ends, respectively.
	
	\tu{(2)} Let $\tb{e}(z)$ be a primitive end and $n\geq 0$ be an integer. Then there is a puzzle piece $P_{n'}(z)$ such that $$P_{n'}(z)\Subset P_n(z).$$ Consequently, $\tb{e}(z)$ is disjoint from the boundary of any puzzle piece; and if $\tb{e}(z)\cap \tb{e}(z')\neq \es$ then $\tb{e}(z)=\tb{e}(z')$.
\end{lem}
\begin{proof}
	(1) Suppose $\tb{e}(z)$ is satellite with $z\in Z_\infty$. Then $Y_n(f(z))=f(Y_n(z))$ for $n\geq 1$. Thus $\tb{e}(f(z))=\bigcap f(Y_n(z))=f(\bigcap Y_n(z))=f(\tb{e}(z))$, where the second equality follows from the fact that $Y_{n+1}(z)\subset Y_n(z)$. Similarly, in the case of a primitive, we have $\tb{e}(f(z))=f(\tb{e}(z))$.
	
	(2) Let $\xi_1,\ldots, \xi_l$ be the points in $\partial P_n(z)\cap J_f$. Then $\xi_i\in Z_\infty$. By definition, $z\notin\tb{e}(\xi_i)$. Thus $\xi_i\notin \partial P_{n_i}(z)$ for some $n_i>n$; for otherwise, $z\in\tb{e}(\xi_i)$, a contradiction. Let $n'=\tu{max}\{n_i\}$. It follows that $\partial P_{n'}(z)\cap \partial P_n(z)\cap J_f=\es$. Hence $P_{n'}(z)\Subset P_{n}(z)$.
	
	Assume by contradiction that $\tb{e}(z)\neq \tb{e}(z')$. Then $P_n(z)\cap \tb{e}(z')=\es$ for some $n$. Thus $\tb{e}(z)\subset \ol{P_{n'}(z)}\subset P_n(z)$ for a large $n'>n$. We have $\tb{e}(z)\cap \tb{e}(z')=\es$, a contradiction.
\end{proof}

We observe that it is possible for two satellite ends to intersect, resulting in a non-trivial continuum.
For example, consider the tuning of the Douady rabbit polynomial using the ``dendrite'' map $g(z)=z^2+\tb{i}$. The resulting map can be denoted by $f(z)=z^2+c_0$. Let $\xi$ be the $\alpha$-fixed point of $f$. Then both satellite ends $\tb{e}(\xi)$ and $\tb{e}(-\xi)$ are formed by three copies of $J_g$. The intersection $\tb{e}(\xi)\cap\tb{e}(-\xi)$ is exactly the copy of $J_g$ containing zero.   

\subsection{Triviality of satellite ends}
Most satellite ends are trivial under our assumption that $f$ admits no irrationally neutral cycles.

\begin{prop}\label{prop:key}
	Let $z\in Z$ with each external ray landing at $z$ being fixed. Then $\tb{e}(z)=\{z\}$.
\end{prop}
\begin{proof}
	Fix a puzzle piece $P_0$ of depth-0 in $Y_0(z)$. It determines a sequence of depth-$n$ puzzle pieces $P_n\subset Y_n(z)$ such that $P_{n+1}\subset P_{n}\subset\cdots\subset P_0$. By condition, we have $f(P_{n+1})=P_n$. Let $K=\bigcap \ol{P_n}$. For the proof, it suffices to show that $K=\{z\}$. 
	
	We assume by contradiction that $K$ contains more than one point. Note that the degree of $f: P_{n+1}\to P_{n}$ eventually stabilizes at a number, say $\delta$, for all $n\geq N$. By Lemma \ref{lem:fixed-points}, $f$ has exactly $\delta$ distinct fixed points in $\ol{P_N}$. Therefore, $K$ contains exactly $\delta$ fixed points. From the choice of $\Gamma$, none of these fixed points is attracting. Thus, they are repelling or parabolic by the assumption that $f$ admits no irrationally neutral cycles.
	
	Let $\xi$ be a fixed point in $K$. Claim that $f$ fixes each external ray that lands at $\xi$. Indeed, it holds clearly if $\xi$ receives just one external ray. In other cases, $\xi\in \partial P_0$. Let $\Delta\subset P_0$ be the disk in Lemma \ref{lem:puzzles} based at $\xi$. Then $\partial\Delta\cap \partial P_0$ contains a segment $\alpha$ from an external ray $R$  that lands at $\xi$. Since $f(P_1)=P_0$, we have $\alpha\subset f(\alpha)$ and so $f(R)\subset R$. The claim is proved.
	
 	Consider a conformal map $h:\mathbb{C}\sm K\to \mathbb{C}\sm \olD.$ According to \cite[Corollary 17.10]{M1}, $\gamma=\ol{h(\partial P_N\sm K)}$ is a Jordan curve outside of $\mb{D}$. Furthermore, the intersection $$X=\gamma\cap \partial \mb{D}=\gamma\sm {h(\partial P_N\sm K)}$$ is a finite set.

	Let $g: P'_{N+1}\to P'_N$ defined by $g(w)=h\circ f\circ h^{-1}(w)$, where $P_N'=h(P_N\sm K)$. Using the Schwartz reflexion principle, the map $g$ can be extended to $\partial \mb{D}\sm X$ and, by continuity, to an orientation preserving selfcovering $G:\partial \mb{D}\to\partial \mb{D}$. By counting preimages of a point where $h^{-1}$ has radial limit, we have $\tu{deg}(G)=\delta$. Note that the image of a fixed ray landing at $K$ under $h$ is $g$-fixed and it then lands at a $G$-fixed point. Thus $G$ has at least $\delta$ fixed points by the above claim. 
	
	We now claim that fixed points of $G$ are repelling. For otherwise, suppose $w\in\partial\mb{D}$ is a one- or two-side attracting $G$-fixed point. In both cases, an attracting domain of $w$ associate to $g$ must intersect $\partial\mb{D}$. Then, in the domain we take a point $w'$ close to $\partial\mb{D}$ such that $g^{n}(w')\in P_{N+1}'$ for all $n\geq 0$. Let $\wt{w}=h^{-1}(w')$. It follows that $\wt{w}\notin K$ and $f^{n}(\wt{w})\subset P_{N+1}$ for all $n\geq 0$. This contradicts to the fact that $K$ is exactly the non-escaping set
	$$\{\xi\in \ol{P_{N+1}}: f^n(\xi)\in \ol{P_{N+1}}\tu{ for all }n\geq 0\}.$$
	The claim is proved.
	
It is well known that a degree-$\delta$ self-covering expansion on the circle is topologically conjugate to the map $x\mapsto x^\delta$ on $\partial \mb{D}$ \cite{Sh}. Therefore, $G$ has $\delta-1$ fixed points. This contradicts the previous count of $\delta$ fixed points. Thus $K$ is a singleton. The proof is complete.
\end{proof} 
\begin{cor}\label{lem:nested1}
	Let $z\in Z$ with each external ray landing at $z$ being fixed. Then for any $n\geq 0$, there is an integer $n'>n$ such that
	
	\tu{(1)} if $z$ is parabolic, $\partial Y_{n'}(z)\cap \partial Y_{n}(z)=\{z\}$; Otherwise,
	
	\tu{(2)} $\partial Y_{n'}(z)\cap \partial Y_n(z)=\es$.
\end{cor}
\begin{proof}
	(1) Note  that $z\in\partial Y_k(z)$ for any $k\geq0$ in this case. Since $\tb{e}(z)=\{z\}$ by Proposition \ref{prop:key}, there is an integer $n'>n$ such that $\partial Y_{n'}(z)$ is disjoint from the finite set $\partial Y_{n}(z)\cap J_f\sm\{z\}$. Therefore, $\partial Y_{n'}(z)\cap\partial Y_n(z)=\{z\}$.
	
	(2) In this case, $Y_n(z)$ contains $z$ in its interior. Since $\tb{e}(z)=\{z\}$ by Proposition \ref{prop:key}, there is an integer $n'>n$ such that $Y_{n'}$ is disjoint from the finitely many points in $\partial Y_n(z)\cap J_f$. Therefore, $Y_{n'}(z)$ belongs to the interior of $Y_n(z)$. 
\end{proof}
\begin{cor}\label{cor:disjoint}
	Let $\tb{e}$ be a non-trivial end. Then $\tb{e}$ is disjoint from the boundaries of bounded Fatou components.
\end{cor}
\begin{proof}
	We assume by contradiction that $\tb{e}\cap\partial\Omega\neq\es$ with $\Omega$ a bounded Fatou component. By some iterate of $f$, we may assume further that $\Omega$ is fixed. Then $\tb{e}$ is contained in a limb of $\Omega$ and $\tb{e}\cap\partial \Omega=\{\xi\}$ is the root of the limb by Theorem \ref{thm:RY}. The Assumption (A2) on $f$ implies that $\xi\in Z_\infty$. By Proposition \ref{prop:key}, $\tb{e}(\xi)=\{\xi\}$. It follows from $\xi\in\tb{e}$ that $\tb{e}=\tb{e}(\xi)$ is a singleton, a contradiction.
\end{proof}

\subsection{Around parabolic satellite ends}
A parabolic point belongs to the boundary of a fixed Fatou component, and receives only fixed external rays, under Assumptions (A1) and (A2), 
\begin{lem}\label{lem:parabolic}
	Let $z\in Z$ be a parabolic point. Let $P_0$ be a depth-0 puzzle piece such that $z\in\partial P_0$. Then there is a pair $(U, V)$ of Jordan disks, with $U\Subset V\subset P_0$, and a positive integer $N_0>0$ with the following properties.
	
		Any primitive end $\tb{\emph{w}} \subset P_{N_0}$ will be mapped into $U$ under some iterate $f^k$, moreover the map
	$$f^k: (U^k, V^k)\to (U, V)$$
	is conformal, where $U^k\Subset V^{k}\subset P_0$ are components of $f^{-k}(U)$ and $f^{-k}(V)$ containing $\tb{\emph{w}}$, respectively. 
\end{lem}
\begin{proof}
	The puzzle piece $P_0$ determines a sequence of depth-$n$ puzzle pieces $P_n$ such that $P_{n+1}\subset P_n$ and $z\in\partial P_n$. Clearly $f(P_{n+1})=P_n$. By Proposition \ref{prop:key}, $\bigcap_n\ol{P_n}=\{z\}$. Thus we may take a large $N$ such that $\ol{P_{N}}$ contains exactly one fixed point $z$. Then $f:P_{N+1}\to P_{N}$ is conformal by Lemma \ref{lem:fixed-points}.
	
	By Corollary \ref{lem:nested1}, we take a large $N_0$ such that $\partial Y_N(z)\cap\partial Y_{N_0}(z)=\{z\}$. Then we obtain a topological quadrilateral $Q=P_{N}\sm \ol{P_{N_0}}$ with its two opposite sides lying in the Fatou set. Thus there are two Jordan disks $U\Subset V$ in $P_0$ such that $Q\cap J_f\subset U$.
	
	Let $\wh{g}=f^{N_0-N}:P_{N_0}\to P_{N}$. Any primitive end $\tb{{w}}\subset P_{N_0}$ is disjoint from boundaries of puzzle pieces, and thus will be eventually iterated into $Q$. Since ends are contained in the Julia set, there is a minimal integer $k$ such that $\wh{g}^k(\tb{{w}})\subset U$. Since $\wh{g}$ is conformal, the two Jordan disks $(U^k, V^k)=\wh{g}^{-k}(U, V)$ is as required. The proof is complete.
\end{proof}

\subsection{Thickened puzzle pieces}
Consider a repelling fixed point $z\in Z$. Let $p\geq 1$ be the period of the external rays landing at $z$. Fix a depth-0 puzzle piece $P_0\subset Y_0(z)$. It determines a sequence of nested depth-$n$ puzzle pieces $P_{n}$ with $P_{n}\subset Y_{n}(z)\cap P_0$. Let $P_n^*=P_{np}$. Then $f^p(P^*_{n+1})=P^*_n$.
	
Following the ``thickening" procedure as described in \cite[Figure 12]{Mlc}, we can generate a sequence of thickened puzzle pieces $\widehat{P}_n$ that contain $P^*_n$. However, in our specific case, it is necessary to ensure that the boundary $\partial \widehat{P}_n$ intersects only finitely many points in the Julia set. To achieve this, we need the following Lemma.
\begin{lem}\label{lem:finitely}
	Any neighborhood $\mc{N}$ of $z$ contains a Jordan curve $\gamma$, surrounding $z$, such that $\gamma\cap J_f$ consists of finitely many preperioidc points.
\end{lem}
\begin{proof}
We may assume $f:\mc{N}\to f(\mc{N})$ is conformal and $\mc{N}\subset f(\mc{N})$. If $\tb{e}(z)$ is a singleton, then for sufficiently large $n$ we have $Y_n(z)\Subset \mc{N}$. We are done by letting $\gamma=\partial Y_n(z)$.  

In other cases, according to \cite[Theorem 2]{Ki}, the \emph{fiber} of $z$ consists of a single point. For the definition of fibers, refer to \cite[Definition 2.5]{Ki}. Consequently, the discussion in the proof of \cite[Corollary 2.16]{Ki} implies that $\mc{N}$ compactly contains a component $X$ of $J_f\sm S$ with $z\in X$, where $S$ is a finite set consisting of preperiodic points in $J_f$. By applying \cite[Lemma 2.17]{Ki}, we obtain a Jordan curve $\gamma'$ such that $\gamma'\cap J_f\subset X$ and $\gamma'\subset \mc{N}$.
\end{proof}
	Now, the \emph{thickened puzzle piece} $\wh{P}_0$ of depth zero is the interior of the union of $\ol{P_0}$ and $\ol{\Delta}$, where $\Delta$ is a Jordan disk bounded by
	
	(1) an arc $\alpha$ in $\partial P_0$ passing through $z$ (formed by two rays in the Fatou set); and
	
	(2) another arc $\beta$ close to $\alpha$ but lying outside of $\ol{P_0}$ that joins endpoints of $\alpha$. 
	
	\noindent By Lemma \ref{lem:finitely}, we can take the thickening part $\Delta$ in such a way that the intersection $\beta\cap J_f$ contains only a finite number of perperiodic points.
By choosing $\beta$ to be sufficiently close to $\alpha$, we can ensure further that $\Delta$ satisfies the additional restriction: for every critical point $c$ and integer $l\geq 0$, it holds that
	$$\{c, f(c),\ldots, f^l(c)\}\subset  \ol{P_0}\Longleftrightarrow \{c, f(c),\ldots, f^l(c)\}\subset \ol{\wh{P}_0}.$$
	
	The \emph{thickened puzzle piece} $\wh{P}_{n}$ of depth $n\geq 1$ is defined as the component of $f^{-np}(\wh{P}_0)$ containing $P^*_{n}$. Note that $P^*_{n}$ is a component of $f^{-np}(P_0)$. By construction, $\wh{P}_n\sm \ol{P^*_{n}}$ is formed by components that are eventually mapped onto $\Delta$ comformally. See figure \ref{fig:thicken}~.
The virtue of these thickened puzzle pieces is the following statement.
\begin{lem}\label{lem:thicken-nest}
	$\bigcap_n\ol{P_{n}^*}=\bigcap_n\ol{\wh{P}_{n}}.$ Moreover, for any integer $n>0$, there is an integer $n'>n$ such that $\wh{P}_{n'}\Subset \wh{P}_{n}$.
\end{lem}
\begin{proof}
	Let $E=\bigcap_n\ol{\wh{P}_{n}}$ and $g:=f^p:\ol{\wh{P}_1}\to\ol{\wh{P}_0}$. Note that $E$ consists of points with their $g$-orbits stay in the closure  of $\wh{P}_1$. Let $\xi\in E$. Then $\xi$ cannot be iterated into $\Delta$ under $g$. This is because points in $\Delta$ will be eventually mapped outside of $\Delta$. Thus $\xi\in \ol{P^*_{n}}$ for all $n\geq 0$. Hence $E\subset \bigcap_n\ol{P^*_n}$. The converse is clearly true.
	
We claim that $E\cap \partial P_0^*=\{z\}$. Note that $E$ is contained in a sector bounded by a pair of rays in $\partial P_0$ that converge to $z$. If $p=1$, $f(E)=E=\{z\}$ by Proposition \ref{prop:key} and it holds clearly. If $p\geq 2$, we have $f(E)\cap E=\{z\}$. Then $E$ does not contain any other fixed points of $f$ except $z$. Thus, $\partial P^*_{n_0}\cap Z=\{z\}$ for a large $n_0$. Note that any pair of $$ f(\partial P^*_{n_0+1}), \ldots, f^p(\partial P^*_{n_0+1})=\partial P^*_{n_0}$$ have no common points in $J_f$ except $z$. Since $\partial P^*_0\cap J_f$ consists of points in $f^{-1}(Z)$, it follows that $\partial P^*_{n_0+1}\cap \partial P^*_0=\{z\}$. This implies $E\cap\partial P^*_0=\{z\}$.

Since a thickened puzzle piece contains $z$ in its interior, we have $E\cap \partial \wh{P}_0=\es$ by the claim.  Consequently, $\wh{P}_N\Subset \wh{P}_0$ for a large integer $N$. By applying the inverse of $g^{n}$, we conclude that $\partial\wh{P}_{n'}\cap \partial \wh{P}_{n}=\es$ with $n'=N+n$. The proof is complete.	
\end{proof}

\begin{figure}[h]
	\begin{tikzpicture}
		\node at (0,0) {\includegraphics[width=12cm]{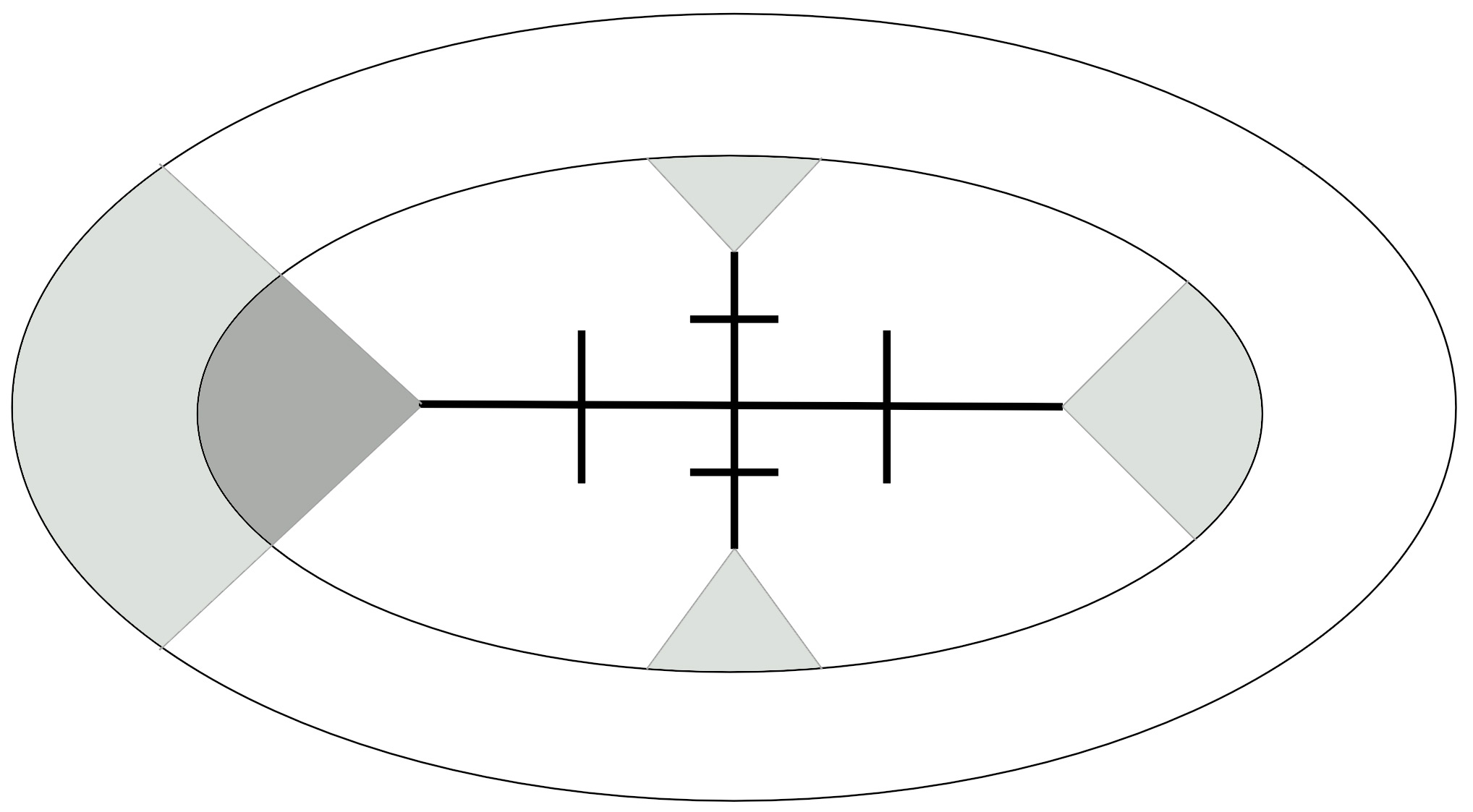}};
		\node at (1.75,1.25) {$P^*_{N}$};
		\node at (2.55,2.25) {$P^*_{N'}$};
		\node at (-5,0) {$D$};
		\node at (5, 0){$W_i$};
	\end{tikzpicture}
	\caption{Illustrating the proof of Lemma \ref{lem:enlarge}. The two ellipses represent the two thickened nested puzzle pieces $\wh{P}_N$ and $\wh{P}_{N'}$, while the shaded area represents the thickened parts $\wh{P}_n\sm P^*_n$ for $n=N, N'$. One of the four components $\wh{P}_N\sm P^*_{N}$ is nested in $D=\wh{P}_{N'}\sm P^*_{N'}$.}
	\label{fig:thicken}
\end{figure}
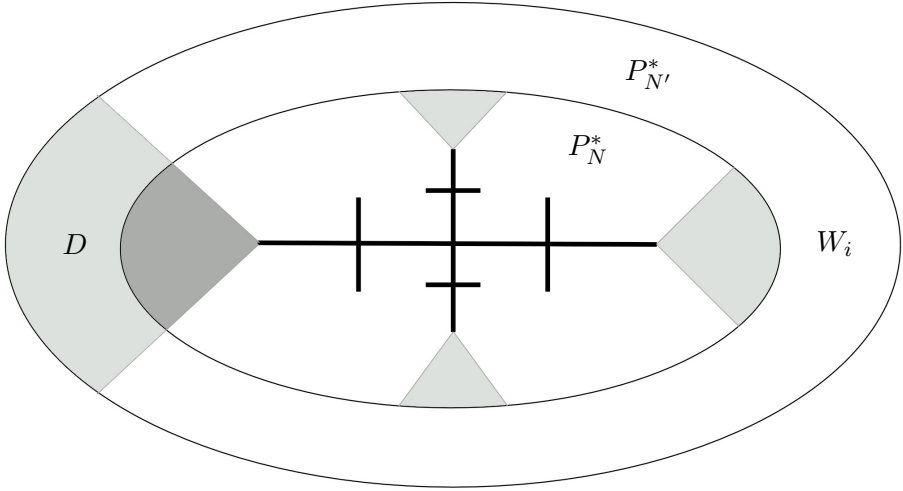
\subsection{Around repelling satellite ends}
We still assume $z\in Z$ is a repelling fixed point and $E=\bigcap_n\ol{\wh{P}_n}$. The next lemma is an analogue of Lemma \ref{lem:parabolic}.
\begin{lem}\label{lem:enlarge}
	There are finitely many pairs $(U_i, V_i)$ of Jordan disks, with $U_i\Subset V_i\subset \wh{P}_0$, and a positive integer $N_0>0$ with the following properties.
	
	Any primitive end $\tb{\emph{w}} \subset P_{N_0}$ will be mapped into some $U_i$ under some iterate $f^k$, moreover the map
	$$f^k: (U_i^k, V_i^k)\to (U_i, V_i)$$
	is conformal, where $U_i^k\Subset V_i^{k}\subset P_0$ are the components of $f^{-k}(U_i)$ and $f^{-k}(V_i)$ containing $\tb{\emph{w}}$, respectively. 
\end{lem}
\begin{proof}
	We may assume $\wh{P}_{N'}\sm E$ is disjoint from critical points of $f$ for a large $N'$. By Lemma \ref{lem:thicken-nest}, we can take a thicken puzzle pieces $\wh{P}_{N}$ that is compactly contained in $\wh{P}_{N'}$. Then we obtain a covering
	$$\wh{g}:=g^{N-N'}: \wh{P}_{N}\sm E\to \wh{P}_{N'}\sm E$$
	between the two annuli. Let $D=\wh{P}_{N'}\sm P^*_{N'}$. Then $\partial D\cap \partial P_{N'}^*$ consists of rays $R_j$ in the Fatou set, together with their landing points. These rays $R_j$ cut the annulus $A:=\wh{P}_{N'}\sm \ol{\wh{P}_N}$ into disjoint quadrilaterals $W_i$. For each $W_i$, we take two Jordan disks $U_i$ and $V_i$ such that $$W_i\cap J_f\subset U_i\Subset V_i\subset \wh{g}(\wh{P}_{N'})\tu{ and }V_i\cap R_j=\es$$ for all $R_j$. This can be done, since $R_j$ are contained in the Fatou set.
	
	Note that any primitive end $\tb{{w}}\subset P_N$ is disjoint from $E$ and from boundaries of puzzle pieces. Thus $\tb{{w}}$ will eventually be mapped under $\wh{g}$ into the annulus $A$. Assume $\tb{{w}}_k=\wh{g}^k(\tb{{w}})\subset A$. Since $\tb{{w}}_k\subset J_f$, some $U_i$ contains $\tb{{w}}_k$. Let $U_i^k$ and $V_i^k$ be the components of $U_i$ and $V_i$ containing $\tb{{w}}_i$. By the Riemann-Hurwitz formula, they are Jordan disks and $\wh{g}^k:(U_i^k, V_i^k)\to (U_i, V_i)$ is conformal. 
	
	By the choice of $\Delta$, a component $D_n$ of $\wh{g}^{-n}(D)$ is either contained in $D$ or disjoint from $D$; and $D_n\Subset P^*_{N'}$ in the latte case. Since $\tb{{w}}\subset P_{N}^*$, the component of the iterated pre-images of $D$ containing $\tb{{w}}_k$, if any, belongs to $P^*_{N'}$, thus $V^k_i\subset P^*_{N'}\subset P_0$. 
	
	Let $P_{N_0}=P_{Np}=P^*_{N}$. The proof is complete.	
\end{proof}
\section{Wandering ends}
The aim of this section is to prove that wandering ends are trivial. We will show that non-persistently recurrent ends satisfy the bounded degree condition, making them trivial according to Lemma \ref{lem:singleton}. While the triviality of persistently recurrent ends follows from the well-known result on the existence of complex bounds. 

There are arguments in \cite{WYZ,RYZ} concerning the triviality of ends, which are similar in spirit to our approach here. However, what sets this paper apart is the existence of non-trivial satellite ends, which brings about a significant distinction. See Lemma \ref{lem:periodic} and Lemma \ref{lem:singleton}.

Throughout this section we still assume that $f$ is a polynomial having connected Julia set and without irrationally neutral cycles, satisfying Assumptions (A1) and (A2).
\subsection{Distortion and shrinking lemmas}
Let $A\subset\mb{C}$ be an annulus with non-degenerated boundary components. Then there exists a confromal map sending $A$ to a standard annulus $\{z\in\mb{C}: 0<r<|z|<1\}$, where $r>0$ is uniquely determined by $A$. The \emph{modulus} of $A$ is defined as $\tu{mod}(A)=\frac{1}{2\pi}\tu{log}(1/r)$, which is invariant under conformal maps.

Consider a Jordan disk $U\subset \mb{C}$ and a point $z\in U$. The \emph{shape} of $U$ about $z$ is defined by
$${\rm Shape}(U,z)=\frac{\sup_{x\in \partial U}|x-z|}{\inf_{x\in \partial U}|x-z|}
.$$
It is clear that ${\rm Shape} (U,z)=1$ if and only if $U$ is a round
disk centered at the point $z$.

\begin{lem}\cite[Lemma 2.2]{RYZ}\label{4:distortionlemma}
	Let $U\Subset V\subset \mathbb{C}$ be Jordan disks with  $\tu{mod}(V\setminus\overline{U})\ge m>0$. Then there is a positive constant $C=C(m)$
	such that for any two points $x,y\in U$,
	$${\rm Shape} (V,y)\leq C\cdot{\rm Shape} (V,x).$$
\end{lem}

The following lemma tells us that shapes by pullback via proper holomorphic maps can be controlled by some constants depending only on degrees and moduli.

\begin{lem}\cite[Lemma 4.7]{RY}\label{2:distortionlemma}
	Let $U_i\Subset V_i\subset\mathbb{C}$, $i\in \{1,2\}$, be Jordan disks with $\tu{mod}(V_2\setminus U_2)\geq m>0$. Suppose that $g:V_1\to V_2$ is
	a proper holomorphic map of degree $\leq\delta$, and $U_1$ be
	a component of $g^{-1}(U_2)$. Then, there is a positive constant $C=C(\delta, m)$ such that for all $z\in U_1$
	$${\rm Shape} (U_1,z)\leq C\cdot{\rm Shape} (U_2,g(z)).$$
\end{lem}
\begin{lem}\label{lem:shrink}
	Let $E\subset J_f$ be a full continuum. Let $U\Subset V$ be Jordan disks such that $f^{n_k}(E)\subset U$ with $n_k\to \infty$ as $k\to\infty$. Suppose the maps $f^{n_k}:V_k\to V$ have bounded degree $\delta$, where $V_k$ are components of $f^{-n_k}(V)$ containing $E$. Then $E$ is a singleton. 
\end{lem}

\begin{proof}
	We assume by contradiction that $\tu{diam}(E)=\lambda>0$. Let $W$ be a Jordan disk such that $U\Subset W\Subset V$. Let $W_k$ be the component of $f^{-n_k}(W)$ containing $E$. By Lemma \ref{4:distortionlemma}, there is a constant $C_1$ depending on $\tu{mod}(W\sm\ol{U})$ such that for all $x, y\in U$
	$$\tu{Shape}(W,y)\leq C_1\cdot\tu{Shape}(W,x).$$
	
	Fix a point $z\in E$. By Lemma \ref{2:distortionlemma}, there is a constant $C_2$ such that $$\tu{Shape}(W_k,z)\leq C_2\cdot\tu{Shape}(W, f^{n_k}(z))$$ for all $k$. 
	Since $f^{n_k}(z)\in U$, it follows from the above two inequalities that $\tu{Shape}(W_k, z)$ are bounded by a constant $C$. Then we have
	$$\tu{dist}(z,\partial W_k)\geq \frac{1}{C}\cdot \frac{\tu{diam}(W_k)}{2}\geq \frac{\lambda}{2C}=\epsilon_0.$$
	Thus the standard Euclidean disk $\mb{D}(z, \epsilon_0)\subset W_k$ for all $k$. Then $f^{n_k}(\mb{D}(z,\epsilon_0))\subset W$ is bounded. It contradicts the fact that $\mb{D}(z,\epsilon_0)$ contains points that escape to infinity.
\end{proof}
\subsection{Bounded degree condition}
An end $\tb{e}$ is called \emph{periodic} if $f^p(\tb{e})=\tb{e}$ for some positive integer $p$; and is called \emph{wandering} if $f^n(\tb{e})\cap f^m(\tb{e})=\es$ for any $n\neq m\geq 0$. By definition, wandering ends are primitive.

\begin{defi}
	A wandering end $\tb{e}$ is said to satisfy the \emph{bounded degree condition}, if there is a puzzle piece $V$ and integers $n_k\to\infty$ as $k\to\infty$ such that 
	$$\tu{deg}(f^{n_k}:V_k\to V)\leq \delta$$
	for all $n_k$, where $V_k$ are components of $f^{-n_k}(V)$ containing $\tb{e}$ and $\delta$ is a constant. 
\end{defi}

\begin{lem}\label{lem:singleton}
	A wandering end $\tb{e}$ satisfying the bounded degree condition is trivial.
\end{lem}
\begin{proof}
	We may assume the puzzle piece $V$ contains infinitely many ends $f^{n_k}(\tb{e})$, and the degree of $f^{n_k}:V_k\to V$ is bounded by $\delta$. If there is a puzzle piece $U\Subset V$ such that $\#\{k: f^{n_k}(\tb{e})\subset U\}=\infty$, we are done by Lemma \ref{lem:shrink}. 
	
	We assume a point $\xi\in \partial V\cap J_f\subset Z_\infty$ satisfies that, for each sufficiently large $n$, the depth-$n$ puzzle piece $P_n\subset Y_n(\xi)\cap V$ contains infinitely many ends from $\{f^{n_k}(\tb{e})\}$. By taking an iterate of $f$ on a high depth puzzle piece in $V$, whose boundary contains $\xi$, we assume further that $\xi$ is fixed and $V\subset Y_0(\xi)$ is a puzzle piece of depth zero.
	
	Let $N_0$ be the integer in Lemmas \ref{lem:parabolic} or \ref{lem:enlarge}. Then by passing to a subsequence of $\{f^{n_k}(\tb{e})\}$, there are integers $m_k$ and two Jordan disks $U'\Subset V'$ such that $f^{m_k+n_k}(\tb{e})\subset U'$ and the maps $f^{m_k}:V'_k\to V'$ are conformal. Here $V_k'\subset V$ are the components of $f^{-m_k}(V')$ containing $f^{n_k}(\tb{e})$. Let $V_k''$ be the components of $f^{-n_k}(V_k')$ containing $\tb{e}$. Then the compositions $f^{m_k+n_k}:V_k''\to V'$ have bounded degree $\delta$ as $V_k'\subset V$. By Lemma \ref{lem:shrink} again, $\tb{e}$ is a singleton. The proof is complete.
\end{proof}

\subsection{Combinatorial accumulation ends}
\begin{defi}
	Let $\tb{e}$ be a wandering end. The orbit of $\tb{e}$ is denoted by $\tu{orb}(\tb{e})=\{\tb{e}, f(\tb{e}),\ldots\}$.  An end $\tb{\emph{w}}(z)$ is called a \emph{combinatorial accumulation end} of the orbit if for any integer $n\geq 0$, the puzzle piece $P_n(z)$ (or $Y_n(z), z\in Z_\infty$, in the satellite case) contains infinitely many elements of $\tu{orb}(\tb{e})$. The collection of such end $\tb{\emph{w}}(z)$ is denoted by $\omega(\tb{{e}})$. 
\end{defi}
 By definition, an end $\tb{{w}}(z)\notin \omega(\tb{e})$ if and only if $P_n(z)$ or $Y_n(z)$ contains at most finitely many elements of $\tu{orb}(\tb{e})$ for any integer $n$.

\begin{lem}\label{lem:periodic}
	Let $\tb{e}$ be a wandering end. 

	\tu{(1)} If $\omega(\tb{e})$ contains a primitive end cycle $\tb{\emph{w}}_1,\ldots, \tb{\emph{w}}_p$, then there is a puzzle piece $P$ disjoint from the cycle such that
	$\#\{k: f^k(\tb{e})\subset P\}=\infty.$	
	
	\tu{(2)} If $\omega(\tb{e})$ contains a fixed satellite end $\tb{e}(z), z\in Z$, then there is a puzzle piece $P$ whose boundary is disjoint from $z$ such that 
	$\#\{k: f^k(\tb{e})\subset P\}=\infty.$
\end{lem}
\begin{proof}
	(1) Let $\mc{P}_n$ be the union of depth-$n$ puzzle pieces that contain one of $\tb{{w}}_i$.
	If the statement is not true, only finitely many elements of $\tu{orb}(\tb{e})$ lie outside $\mc{P}_n$ for each $n$. Thus, there are integers $n_0$ and $q$ such that $\mc{P}_{n_0}$ contains all elements of $\tu{orb}(f^q(\tb{e}))$. We may assume $n_0$ is sufficiently large such that $\mc{P}_n$ is formed by $p$ disjoint depth-$n$ puzzle pieces and the degree of $f:\mc{P}_{n}\to \mc{P}_{n-1}$ is a constant for each $n\geq n_0$. We claim that
	\begin{equation}\label{eq:escape}
	\bigcap \mc{P}_n=\{\xi\in \mc{P}_{n_0}: f^k(\xi)\in \mc{P}_{n_0}\tu{ for all }k\geq 0\}=:E.
	\end{equation}
	Indeed, let $\xi\in E$ and $\xi_k=f^k(\xi)$. Fix an integer $k\geq 1$. Assume $\xi_k\in P_{n_0}\subset \mc{P}_{n_0}$. Then the preimage $(f|_{\mc{P}_{n_0}})^{-1}(P_{n_0})$ has just one component in $\mc{P}_{n_0}$, which is a puzzle piece of depth $n_0+1$ in $\mc{P}_{n_0+1}$. Since $\xi_{k-1}\in \mc{P}_{n_0}$, we have $\xi_{k-1}\in \mc{P}_{n_0+1}$. By successively processing the discussion on $\xi_{k-1},\ldots, \xi_{1}$, we conclude that $\xi\in\mc{P}_{n_0+k}$. Thus $\xi\in\bigcap \mc{P}_n$.
	
	It follows from the claim that $f^q(\tb{e})$ is contained in $\bigcap \mc{P}_n$, which is the union of $\tb{{w}}_1,\ldots,\tb{{w}}_p$. This contradicts the condition that $\tb{e}$ is wandering.
	
	(2) The proof is similar as above. In this case, let $\mc{P}_n$ be the union of depth-$n$ puzzle pieces whose boundaries contain $z$. If the statement is not true, then elements of $\tu{orb}(f^q(\tb{e}))$ are contained in $\mc{P}_{n_0}$ for some integers $n_0$ and $q$. Assume $n_0$ is large enough such that the degree of $f:P_n\to P_{n-1}$ stabilizes for each $n\geq n_0$, where $P_k$ are the depth-$k$ puzzle pieces in $\mc{P}_k$. Based on the same discussion as above, the claim \eqref{eq:escape} holds as well in this case. Therefore, $f^q(\tb{e})$ belongs to $\bigcap\mc{P}_n$, which is a subset of $\tb{e}(z)$. Once again, this contradicts the fact that $\tb{e}$ is wandering.
\end{proof}

\subsection{Triviality of wandering ends}
We shall prove that non-persistently recurrent ends have bounded degree condition; See Lemmas \ref{lem:pperiodic}, \ref{lem:non-recurrent}, \ref{lem:crit} and \ref{lem:successor}.
\begin{defi}
	Let $\tb{e}$ be a wandering end and $P$ be a puzzle piece. The \emph{first entry time} $r_{\tb{e}}(P)$ of $\tb{e}$ into $P$ is the minimal positive integer $k$ such that $f^{k}(\tb{e})\subset P$. We denote by $L_{\tb{e}}(P)$ the unique puzzle piece containing $\tb{e}$ such that $f^{r_{\tb{e}}(P)}(L_{\tb{e}}(P))=P$.
\end{defi}

We denote by $\kappa$ the number of critical points in ${J}_f$.
\begin{lem}\label{lem:first}
	Let $\tb{e}$ be a wandering end. Let $P$ be a puzzle piece such that the first entry time $r=r_{\tb{e}}(P)$ exists. Then the degree of $f^r: L_\tb{e}(P)\to P$ is bounded by $d^\kappa$.
\end{lem}
\begin{proof}
	Let $Q_k:=f^k(L_{\tb{e}}(P)), 0\leq k\leq r$. Claim that the $r$ puzzle pieces $Q_0,\cdots Q_{r-1}$ are pairwise disjoint. For otherwise, assume $Q_{r_1}\cap Q_{r_2}\neq \emptyset$ with $r_1<r_2$. Then $f^{r_1}(\tb{e})\in Q_{r_1}\subset Q_{r_2}$. We pull back $Q_{r_2}$ by $f^{r_1}$ at $\tb{e}$, and then obtain a puzzle piece $W$ containing $\tb{e}$ whose depth is strictly less than that of $Q_0$. Therefore, $f^{r_1+r-r_2}(W)=P$ with $r_1+r-r_2<r_\tb{e}(P)$. It is a contradiction. 
	
	By the claim, each critical point of $f$ appears at most once in the orbit $Q_0,\cdots, Q_{r-1}$. Thus, the degree of $f^r:L_\tb{e}(P)\to P$ is bounded by $d^\kappa$. 
\end{proof}

\begin{lem}\label{lem:pperiodic}
	Let $\tb{e}$ be a wandering end. If $\omega(\tb{e})$ contains a preperiodic end, then $\tb{e}$ satisfies the bounded degree condition.
\end{lem}
\begin{proof}
	Since $f(\omega(\tb{e}))\subset \omega(\tb{e})$ by definition, we may assume $\tb{{w}}\in \omega(\tb{e})$ is periodic. We first consider the case that $\tb{{w}}$ is primitive. By Lemma \ref{lem:periodic} (1), there is a puzzle piece $U$, disjoint from elements of $\tu{orb}(\tb{{w}})$, such that $\#\{k: f^k(\tb{e})\subset U\}=\infty$.
	
	Let $P_n$ be the depth-$n$ puzzle piece containing $\tb{{w}}$. Then $P_n$ contains at least an element of $\tu{orb}(\tb{e})$. Let $r_n$ be the first entry time of $\tb{e}$ into $P_n$. Since $\tb{{w}}\notin \tu{orb}(\tb{e})$, it implies that $r_n\to\infty$ as $n\to\infty$. The end $f^{r_n}(\tb{e})$ will eventually run into the puzzle piece $U$. The first entry time of $f^{r_n}(\tb{e})$ into $U$ is denoted by $s_n$. Let $Q_n=L_{f^{r_n}(\tb{e})}(U)$. Since $f^{r_n}(\tb{e})\subset Q_n\cap P_n$, we have $Q_n\subset P_n$ or $P_n\subset Q_n$. The latter case cannot happen, for otherwise $f^{s_n}(\tb{{w}})\subset f^{s_n}(P_n)\subset f^{s_n}(Q_n)=U$, contradicting the assumption on $U$. By Lemma \ref{lem:first}, both maps 
	\begin{equation}\label{eq:composition}
	f^{r_n}:L_{\tb{e}}(P_n)\to P_n\tu{ and }f^{s_n}:Q_n\to U
	\end{equation}
 	have bounded degree $d^\kappa$. Their composition $f^{r_n+s_n}: Q_n'\to U$ has degree less than $d^{2\kappa}$, where $Q_n'\subset L_{\tb{e}}(P_n)$ is the component of $f^{-r_n}(Q_n)$ containing $\tb{e}$. 
	
	In the case that $\tb{{w}}$ is satellite, we assume $\tb{{w}}=\tb{e}(z), z\in Z$. Let $U$ be the puzzle piece in Lemma \ref{lem:periodic}\,(2) satisfying that $z\notin \ol{U}$ and $\#\{k:f^{k}(\tb{e})\subset U\}=\infty$. For any integer $n$, there is a depth-$n$ puzzle piece $P_n\subset Y_n(z)$ containing $f^{r_n}(\tb{e})$ with $r_n$ the first entry time of $\tb{e}$ into $P_n$. It is clear that $r_n\to\infty$ as $n\to\infty$. By the same process as above, we obtain an integer $s_n$ and a puzzle piece $Q_n$ that contains $f^{r_n}(\tb{e})$. Then either $Q_n\subset P_n$ or $P_n\subset Q_n$. Again, the latter case implies that $z=f^{s_n}(z)\in f^{s_n}(\ol{P_n})\subset f^{s_n}(\ol{Q_n})=\ol{U}$, impossible. Therefore, by substituting $P_n$ with $Q_n$, the composition of the two maps in \eqref{eq:composition} has bounded degree $d^{2\kappa}$ as well. The proof is complete.  
\end{proof}

\begin{lem}\label{lem:non-recurrent}
	Let $\tb{e}$ be a wandering end. If there is a wandering end $\tb{\emph{w}}\in\omega(\tb{e})$ such that $\omega(\tb{\emph{w}})\neq \omega(\tb{e})$, then $\tb{e}$ satisfies the bounded degree condition.
\end{lem}
\begin{proof}
	By Lemma \ref{lem:pperiodic} we may assume ends in $\omega(\tb{e})$ are wandering. Since wandering ends are primitive, it follows from Lemma \ref{lem:ends} (2) and the definition that $\omega(\tb{{w}})\subset \omega(\tb{e})$.
	
	 We take a wandering end $\tb{{w}}'\in\omega(\tb{e})\setminus\omega(\tb{{w}})$. This implies that there is a puzzle piece $U$, containing $\tb{{w}}'$, of sufficiently high depth, such that
	 $U$ is disjoint from elements in $\tu{orb}(\tb{{w}})$ but contains infinitely many elements of $\tu{orb}(\tb{e})$.
	 
	 For any depth-$n$ puzzle piece $P_n$ containing $\tb{{w}}$, let $r_n$ be the first entry time of $\tb{e}$ into $P_n$. Since $\omega(\tb{e})=\omega(f^k(\tb{e}))$ for any $k\geq 1$, we have $\tb{{w}}\notin \tu{orb}(\tb{e})$. Thus $r_n\to\infty$ as $n\to\infty$.
	  It follows from the choice of $U$ that $f^{r_n}(\tb{e})$ eventually runs into $U$; the first entry time $s_n$ exists.

	  Let $Q_n=L_{f^{r_n}(\tb{e})}(U)$. Since $f^{r_n}(\tb{e})\subset Q_n\cap P_n$, either $Q_n\subset P_n$ or $P_n\subset Q_n$. The latter case implies that $\tb{{w}}\subset Q_n$, and then $f^{s_n}(\tb{{w}})\subset f^{s_n}(Q_n)=U$, a contradiction. 
	  
	 Let $Q_n'\subset L_{\tb{e}}(P_n)$ be the component of $f^{-r_n}(Q_n)$ containing $\tb{e}$. By Lemma \ref{lem:first}, the composition of the two maps $f^{r_n}:Q_n'\to Q_n\subset P_n$ and $f^{s_n}:Q_n\to U$ has bounded degree $d^{2\kappa}$. The proof is complete.
\end{proof}

An end is called \emph{critical} if it contains a critical point of $f$. The collection of critical ends in $\omega(\tb{e})$ is denoted by $\omega\tu{Crit}(\tb{e})$.

\begin{lem}\label{lem:crit}
	Let $\tb{e}$ be a wandering end. If $\omega\tu{Crit}(\tb{e})=\emptyset$, then 
	$\tb{e}$ satisfies the bounded degree condition. 
\end{lem}
\begin{proof}
	We may assume $\tu{orb}(\tb{e})$ contains no critical ends. Let $\tb{{w}}_1,\ldots \tb{{w}}_m$ be the critical ends of $f$. If $\tb{{w}}_i$ is primitive, there is an integer $n_i$ such that the depth-$n$ puzzle piece $P_n$ containing $\tb{{w}}_i$ is disjoint from elements of $\tu{orb}(\tb{e})$ for all $n\geq n_i$. If $\tb{{w}}_i=\tb{e}(z_i)$ is satellite with $z_i\in Z_\infty$, then $Y_{n}(z_i)$ is disjoint from elements of $\tu{orb}(\tb{e})$ for all $n\geq n_i$. 
	
	Let $N=\tu{max}\{n_i: 1\leq i\leq m\}$. Take a puzzle piece $U$ of depth $N$ such that $f^{n_k}(\tb{e})\subset U$ with $n_k\to\infty$ as $k\to\infty$. Then the map $f^{n_k}:Q_{k}\to U$ is conformal, where $Q_k$ is the component of $f^{-n_k}(U)$ containing $\tb{e}$. Indeed, $f^{i}(Q_k)$ is a puzzle piece of depth $N+n_i-i$ containing $f^i(\tb{e})$ for $0\leq i<n_i$; thus it is disjoint from critical ends by the choice of $N$. Therefore, $\tb{e}$ satisfies the bounded degree condition. 
\end{proof}

\begin{defi}
	Let $\tb{e}$ be a wandering critical end. Let $P_n$ be the depth-$n$ puzzle piece containing $\tb{e}$. Then $P_{n+k}$ is called a \emph{successor} of $P_n$ if $f^k(P_{n+k})=P_n$ and the itinerary $P_{n+k}, f(P_{n+k}),\ldots, f^k(P_{n+k})=P_n$
	meets each critical end of $f$ at most twice. 
\end{defi}

 By definition, the degree of  $f^k:P_{n+k}\to P_n$ is bounded by $d^{2\kappa-1}$.

\begin{lem}\label{lem:successor}
	Let $\tb{e}$ be a wandering end. If there is a critical end $\tb{{w}}\in\omega\tu{Crit}(\tb{e})$ and an integer $n_0$ such that $P_{n_0}$, containing $\tb{{w}}$, has infinitely many successors, then $\tb{e}$ satisfies the bounded degree condition.
\end{lem}
\begin{proof}
	Let $P_{n_0+n_k}$ be the successors of $P_{n_0}$ with $n_k\to\infty$ as $k\to\infty$. By definition, the degrees of maps
	$f^{n_k}:P_{n_0+n_k}\to P_{n_0}$
	are bounded by $d^{2\kappa-1}$. Let $r_k$ be the first entry time of $\tb{e}$ into $P_{n_0+n_k}$ and $V_k$ be the component of $f^{-r_k}(P_{n_0+n_k})$ containing $\tb{e}$. Then the composition $f^{n_k+r_k}: V_k\to P_{n_0}$ has degree less than $d^{3\kappa-1}$. The proof is complete.
\end{proof}

A wandering end $\tb{e}$ is called \emph{persistently recurrent} if it satisfies that

(1) $\omega\tu{Crit}(\tb{e})\neq\es$;
 
(2) any end $\tb{{w}}\in\omega(\tb{e})$ is wandering with $\omega(\tb{{w}})=\omega(\tb{e})$;

(3) any depth-$n$ puzzle piece $P_n$, containing a critical end of $\omega\tu{Crit}(\tb{e})$, has only finitely many successors.

The collection of persistently recurrent ends is denoted by $\mc{E}$. Combining Lemmas \ref{lem:singleton}, \ref{lem:pperiodic}, \ref{lem:non-recurrent}, \ref{lem:crit} and \ref{lem:successor}, we have the following corollary.
\begin{cor}\label{cor:non-recurrent}
	Any wandering end not contained in $\mc{E}$ is trivial.
\end{cor}

We do not know whether a persistently recurrent end $\tb{e}$ satisfies the bound degree condition. However, there is a sequence of principal nested puzzle pieces $P_{n_k}$ containing $\tb{e}$ such that $$\tu{mod}(P_{n_{k+1}}\sm \ol{P_{n_k}})\geq \mu,$$ for all $k$ and a positive constant $\mu$. The construction of the principal nest is attributed to Kahn and Lyubich \cite{KL2} in the unicritical case and to Kozlovski-Shen-van Strien \cite{KSS} in the multicritical case. The complex bounds are proven by Kahn-Lyubich \cite{KL1,KL2} for the unicritical case and by Kozlovski-van Strien \cite{KS} and Qiu-Yin \cite{QY} independently for the multicritical case. Interested readers can refer to these references as well as \cite{CDKS} for details. Consequently, we have

\begin{thm}\label{thm:recurrent}
	Any end in $\mc{E}$ is trivial.
\end{thm}

\section{Proofs of Theorems \ref{th:Julia} and \ref{th:Fatou}}
\begin{proof}[Proof of Theorem \ref{th:Julia}]
	Let $J_\infty$ be the Hausdorff limit of a sequence of distinct Julia components $J_n$. Assume by contradiction that $J_\infty$ is not a singleton. Clearly $J_\infty\subset J_f$ and $J_\infty$ is contained in a component $K$ of the filled Julia set. Then by \cite[Main Theorem]{QY}, $K$ is eventually periodic. By considering some iterate of $J_n$ and $J_\infty$, we may assume that $f^p(K)=K$.
	
	Take a disk $D$ bounded by an equipotential curve of the Green's function of the basin of infinity, such that $K\subset D$ and $D\cap \tu{crit}(f)=K\cap \tu{crit}(f)$.  Let $D'$ be the component of $f^{-p}(D)$ containing $K$. Then $\ol{D'}\subset D$.
	
	Claim that there is an arc $\gamma$ in the annulus $H=D\sm\ol{D'}$ joining $\partial D$ and $\partial D'$ such that it is disjoint from the filled Julia $K_f$. Indeed, since $K_f\cap H$ is compact, for each sufficiently small $\epsilon>0$, it is covered by finitely many open Euclidean round disks in $H$ whose centers are in $K_f$ and diameters equal $\epsilon$. Let $E_\epsilon$ be the closure of the union of these disks. Then $K_f\cap H=\bigcap_\epsilon E_\epsilon$. Since $\partial H$ can not be separated by components of $K_f$, for a small $\epsilon_0$ the set $E_{\epsilon_0}$ does not separate $\partial H$. Then one can find an arc $\gamma\subset H$ avoiding the finitely many components of $E_{\epsilon_0}$ joining $\partial D$ and $\partial D'$.
	
	Consider the covering $g=f^p:D'\sm K\to D\sm K$ whose degree is $k\geq 2$. Then $g$ is expanding with respect to the hyperbolic metric $\rho$ on $D\sm K$. Precisely, there exists $\delta>0$ such that  
	$$\rho(z,z')\leq \rho(g(z),g(z'))$$
	whenever $\rho(z,z')<\delta$.
	Let $D''=g^{-1}(D')$. The arcs in $g^{-1}(\gamma)$ cut the annulus $D'\sm \ol{D''}$ into $k$ Jordan domains $W_i$. Let $M$ be the maximal $\rho$-diameters among these domains.  Note that $g^n$ sends any component of $g^{-n}(W_i)$ homeomorphically onto $W_i$ by the Riemann-Hurwitz formula. It follows from the expansion of $g$ that $\rho\tu{-diameter}(W)\leq M,$ where $W$ runs over all components of $g^{-n}(W_i)$ for all $n\geq 0$ and $1\leq i\leq k$. 
	
	All $\{J_n\}$ except finitely many ones are contained in $D''$. Suppose $J_n\subset D''$. Then $J_n$ is eventually iterated into one of $W_i$ under $g$. It follows that $J_n$ is contained in a component the iterated preimage of some $W_i$. Thus, we have $\rho\tu{-diameter}(J_n)\leq M$. 
	
	Since $\{J_n\}$ converges to $\partial K$, the Euclidean diameters of $J_n$ tend to $0$ as $n\to\infty$. Then $J_\infty$ must be a singleton, a contradiction. The proof is complete.	
\end{proof}

\begin{proof}[Proof of Theorem \ref{th:Fatou}]
	According to Theorem \ref{th:Julia}. We may assume $J_f$ is connected and $f$ satisfies Assumptions (A1) and (A2) in Subsection \ref{subsection}. Let $\{\Omega_i\}$ be a sequence of distinct bounded Fatou components. Suppose $\ol{\Omega_i}$ Hausdorff converges to a limit $K$. Then $K\subset J_f$. For the proof, it suffices to show that $K$ is a singleton. We argue by contradiction and assume it is not true.
	
	Let $z\in K\sm Z_\infty$ and $P_n$ be the depth-$n$ puzzle piece containing $z$. Denote by $E=\bigcap\ol{P_n}$. We claim that $K\subset E$. Indeed, there are at most finitely many Fatou components whose boundaries intersect $\partial P_n$. Thus all $\Omega_i$ except finitely many ones are contained in $P_n$. This implies that $K\subset E$.
	
	By definition, $E$ is contained in an end $\tb{e}$, which is satellite or primitive. It follows from Corollary \ref{cor:non-recurrent} and Theorem \ref{thm:recurrent} that $\tb{e}$ can not be wandering. Thus, by applying an iterate of $f$, we may assume that $\tb{e}$ is periodic. 
	
	We first consider the case that $\tb{e}$ is primitive. In this case, $\tb{e}=E$. By Lemma \ref{lem:ends}\,(2), take two puzzle pieces such that $P_{N}\Subset P_{N'}$ and
	$f^{N-N'}(P_N)=P_{N'}$. Assuming that $N'$ is sufficiently large, we can consider the map $$g=f^{N-N'}:P_N\setminus \tb{e}\to P_{N'}\setminus\tb{e}$$ which acts as a covering between the two annuli.
	Let $\mc{K}$ be the collection of depth-$N$ puzzle pieces in the annuli $A=P_{N'}\sm \ol{P_N}$. Let $\mc{U}$ be the collection of Fatou components in $P_N$ whose closures intersect either $\partial P_N$ or some $\partial P$ with $P\in \mc{K}$. It is clear that $\mc{K}\cup \mc{U}$ has finitely many elements. 
	
	Each $\ol{\Omega_i}\subset P_{N}$ under an iterate of $g$ will either run into the annulus $A=P_{N'}\sm \ol{P_{N}}$ or intersect with $\partial P_N$. The first time that $\ol{\Omega_i}$ meets $A\cup\partial P_N$ is denoted by $r_i$. Then $g^{r_i}(\ol{\Omega_i})\subset \ol{Q}$ for some $Q\subset \mc{K}\cup \mc{U}$. By the Riemann-Hurwitz formula, $g^{r_i}:Q_i\to Q$ is conformal, where $Q_i$ is the component of $g^{-r_i}(Q_i)$ containing $\Omega_i$.
	
Note that boundaries of bounded Fatou components are disjoint from $\tb{e}$ by Corollary \ref{cor:disjoint}. Thus, the maximum $\rho$-diameters of elements in $\mc{K}\cup \mc{U}$ are bounded by some constant $C$, where $\rho$ is the hyperbolic metric on $D\sm \tb{e}$ with $D$ a Jordan disk containing $\ol{P_{N'}}$. 
	
	Since $g$ is expanding with respect to the metric $\rho$, it follows that $\rho\tu{-diameter}(\ol{\Omega_i})\leq C$ for all $\ol{\Omega_i}\subset P_N$. This implies that the Euclidean diameters of $\ol{\Omega_i}$ tend to zero as $i\to \infty$. Therefore, $K$ is a singleton, contradicting our assumption. 
	
	We are reduced to the case that $\tb{e}$ is satellite. The approach is similar. Assume $\tb{e}=\tb{e}(\xi)$ with $\xi\in Z$. By Proposition \ref{prop:key}, $\xi\in E$ is repelling and $f^p(E)=E$ with $p\geq 2$. By Lemma \ref{lem:thicken-nest} we obtain two thicken puzzle pieces $\wh{P}_{N}\Subset \wh{P}_{N'}$ containing $E$ in their interiors such that 
	$$g=f^{p(N-N')}:\wh{P}_N\sm E\to \wh{P}_{N'}\sm E$$ is covering between the two annuli. In this case, the annulus $A=\wh{P}_{N'}\sm\wh{P}_N$ is divided by $\partial P^*_{N'}$ into finitely many quadrilaterals, whose collection is denoted by $\mc{K}$. By construction, only finitely many Fatou components whose closures intersect $\partial \wh{P}_{N}$. The collection of them is denoted by $\mc{U}$.   
	
	For each $\ol{\Omega_i}\subset \wh{P}_N$, the first entrance $f^{r_i}(\ol{\Omega_i})$ into $A\cup \partial \wh{P}_N$ is either contained in the closure of an element of $\mc{U}$ or intersects some $Q\in \mc{K}$. In the latter case, since the opposite sides, disjoint from $\partial A$, are formed by rays in the Fatou set, we have $f^{r_i}(\ol{\Omega_i})\subset \ol{Q_i}$. Again, by the expansion of $g$ with respect to the hyperbolic metric $\rho$ on $D\sm \tb{e}$, we conclude that $\rho$-diameter($\ol{\Omega_i})\leq C$. Thus, the Euclidean diameters of $\ol{\Omega_i}$ tend to zero as $i\to\infty$. This contradicts the assumption that $K$ is not a singleton. The proof is complete.
\end{proof}

\vskip 1 em \noindent{\it Acknowledgment}: 
The authors would like very much to thank Yongcheng Yin for helpful discussion. This research is partially supported by the NSF of China.


\begin{thebibliography}{1}
	%
	\frenchspacing
	
	
	
	\bibitem[CST]{CST} T. Clark, S. van Strien, and S. Trejo, {\rm Complex bounds for real maps}, Comm. Math.
	Phys. 355 (2017), no. 3, 1001?1119.
	\bibitem[CDKS]{CDKS}
	T. Clark, K. Drach, O. Kozlovski, S. van Strien. {\em The dynamics of complex box mappings}.
	Arnold Math. J. \tb{8}(2022), no. 2, 319--410.

	
	\bibitem[Du]{Du} D. Dudko.
	{\em The decoration theorem for Mandelbrot and Multibrot sets},
	International Mathematics Research Notices, \tb{13}(2017), 3985--4028.
	
	\bibitem[DH1]{DH1} A. Douady and J. H. Hubbard. \emph{Etude dynamique des polynomes complexes}. Publications math\'{e}matiques d'Orsay, 1984.
	
	\bibitem[DH2]{DH2} A. Douady and J. H. Hubbard. \emph{On the dynamics of polynomial-like mappings}. Ann. Sci. \'{E}c. Norm. Sup\'er. \tb{18}(1995), 287--343.
	
	\bibitem[DS]{DS} K. Drach and D. Schleicher. {\em Rigidity of Newton dynamics}, Adv. Math. {\bf 408} (2022), 108591, 64pp.
	\bibitem[FY]{FY} Y. Fu and F. Yang. {\em Mating Siegel and parabolic quadratic polynomials}, arXiv: 2305.15180v1, 2023.
	
	
	
	\bibitem[GM]{GM} L.R. Goldberg, J. Milnor, \emph{Fixed point portraits of polynomial maps, Part II: Fixed Point Portraits}, Ann. Sci. \'{E}cole Norm. Sup. Paris $4^\tu{e}$ s\'erie 26(1993) 51-98.
	
	
	
	
	
	\bibitem[Ki]{Ki} J. Kiwi. \emph{$\mb{R}$eal laminations and the topological dynamics of complex polynomials}. Advances in Mathematics, \tb{184}(2014), 207--267.
	
	
	\bibitem[KL1] {KL1} J. Kahn and M. Lyubich. {\em The quasi-additivity law in conformal
		geometry}. Ann. of Math., 169 (2009), No. 2, 561--593.
	
	
	\bibitem[KL2] {KL2} J. Kahn and M. Lyubich. {\em Local connectivity of Julia sets for unicritical polynomials}. Ann. of Math., 170 (2009), 413--426.
	
	\bibitem[KS] {KS} O. Kozlovski and S. van Strien.
	{\em Local connectivity and quasiconformal rigidity of non-renormalizable  polynomials}.
	Proceedings of London Math. Soc., {\bf 99}(2009), 275--296.
	
	\bibitem[KSS] {KSS} O. Kozlovski, W. Shen and S. van Strien. {\em Rigidity for
		real polynomials}. Ann. of Math., {\bf 165}(2007), 749--841.
	
	
	
	
	\bibitem[Le]{Le} G. Levin, {\em On backward stability of holomorhic dynamical systems}, Fund. Math. 158 (1998) 97--107
	\bibitem[LG]{LG} G. Levin and S. van Strien. {\em Local connectivity of the Julia set of real polynomials.} Ann. of Math. {\bf 2}(1998), 147:471--541.
	
	
	
	
	\bibitem[Mc]{Mc} C. McMullen. 
	\emph{Automorphisms of rational maps}. pp. 31--60 of ``Holomorphic Functions and Moduli", edit. Drasin et al., Springer 1988.
	
	\bibitem[M1]{M1} J. Milnor.
	{\em Dynamics in one complex variable}.
	Vieweg 1999, Third edition 2006.
	
	\bibitem[M2]{Mlc} J. Milnor.
	{\em Local connectivity of Julia Sets\,: Expository Lectures}, pp.67-116 of ``The Mandelbrot set, Theme and Variations'' edit\,: Tan Lei, LMS Lecture Note Series 274, Cambr. U. Press 2000.
	
	\bibitem[Pe]{Pe} C. L. Petersen, {\em Local connectivity of some Julia sets containing a circle with an irrational rotation}, Acta Math. 177 (1996), no. 2, 163?224.
	
	\bibitem[PR]{PR} C. L. Petersen and P. Roesch.
	{\em Carrots for dessert}, Ergod. Th. \& Dynam. Sys. \tb{32}(2012), 2025--2055.
	
	
	\bibitem[QWY] {QWY} W. Qiu, X. Wang and Y. Yin. {\em Dynamics of McMullen maps,} Adv. Math. \textbf{229} (2012), 2525-2577.
	
	\bibitem[QY] {QY} W. Qiu and Y. Yin.
	{\em Proof of the Branner-Hubbard conjecture on Cantor Julia sets}.
	Science in China, Series A, {\bf 52}(2009), No.1, 45-65.
	
	\bibitem[Roe]{Roe} P. Roesch. {\em On local connectivity for the Julia set of rational maps: Newton's famous example}. Ann. Math. \textbf{168}(2008), 1-48.
	
	\bibitem[RY]{RY}P. Roesch and Y. Yin. {\em Bounded critical Fatou components are Jordan domains for polynomials}. Sci. China Math. \textbf{65}(2022), 331-358.
	
	\bibitem[RYZ]{RYZ}P. Roesch, Y. Yin, J. Zeng. {\em Rigidity of non-renormalizable Newton maps}. Sci. China Math. 2022, 65, https://dog.org/10.1007/s11425-016-5135-4
	
	\bibitem[Sh]{Sh} M. Shub. {\em Endomorphisms of compact differentiable manifolds}. Amer. J. Math., \tb{XCI}(1969), 175--199.
	
	
	
	
	
\bibitem[WYZZ]{WYZZ} S. Wang, F. Yang, G. Zhang, and Y. Zhang, {\em Local connectivity for the Julia sets of rational
maps with Siegel disks}, arXiv: 2106.07450v4, 2022.
	
	\bibitem[WYZ]{WYZ} X. Wang, Y. Yin and J. Zeng.
	{\em Dynamics of Newton maps}. \emph{Ergodic theory and Dynamical Systems}, \tb{43}(2023), no. 3, 1035--1080..
	
	\bibitem[Ya]{Ya} J. Yang, {\rm Local connectivity of polynomial Julia sets at bounded type Siegel boundaries}, arXiv: 2010.14003v3, 2021.
	
	\bibitem[YZ]{YZ} Y. Yin and Y. Zhai.
	{\em No invariant line fields on Cantor Julia sets}. Forum Math., {\bf 22}(2010), No.1, 75-94.
	
	\bibitem[Yo]{Yoccoz} J. C. Yoccoz.
	{\em On the local connectivity of the Mandelbrot set}. Unpublished,
	1990.
	
\end{thebibliography}
\end{document}